\documentclass[12pt]{amsart}

\usepackage{enumerate, amsmath, amsthm, amsfonts, amssymb, mathrsfs, graphicx, paralist}
\usepackage[margin=1in]{geometry} 
\usepackage[bookmarks, colorlinks=true, linkcolor=blue, citecolor=blue, urlcolor=blue]{hyperref}

\numberwithin{equation}{section}
\newtheorem{theorem}[equation]{Theorem}
\newtheorem{proposition}[equation]{Proposition}
\newtheorem{lemma}[equation]{Lemma}
\newtheorem{corollary}[equation]{Corollary}

\theoremstyle{definition}
\newtheorem{rmk}[equation]{Remark}
\newenvironment{remark}[1][]{\begin{rmk}[#1] \pushQED{\qed}}{\popQED \end{rmk}}
\newtheorem{eg}[equation]{Example}

\newtheorem{defn}[equation]{Definition}
\newenvironment{definition}[1][]{\begin{defn}[#1]\pushQED{\qed}}{\popQED \end{defn}}

\newcommand{\bF}{\mathbf{F}}
\newcommand{\bk}{\mathbf{k}}
\newcommand{\bQ}{\mathbf{Q}}
\newcommand{\bv}{\mathbf{v}}
\newcommand{\bV}{\mathbf{V}}
\newcommand{\bZ}{\mathbf{Z}}

\newcommand{\rB}{\mathrm{B}}
\newcommand{\rH}{\mathrm{H}}

\newcommand{\arxiv}[1]{\href{http://arxiv.org/abs/#1}{{\tt arXiv:#1}}}
\DeclareMathOperator{\Tor}{Tor}

\begin{document}

\title[The cone of Betti tables over three non-collinear points in the plane]{The cone of Betti tables over three\\ non-collinear points in the plane}


\author{Iulia Gheorghita} 
\address{California Institute of Technology, Pasadena, CA}
\email{\href{mailto:igheorgh@caltech.edu}{igheorgh@caltech.edu}}

\author{Steven V Sam}
\address{University of California, Berkeley, CA}
\email{\href{mailto:svs@math.berkeley.edu}{svs@math.berkeley.edu}}
\urladdr{\url{http://math.berkeley.edu/~svs/}}

\date{December 22, 2014}

\maketitle

\begin{abstract}
We describe the cone of Betti tables of all finitely generated graded modules over the homogeneous coordinate ring of three non-collinear points in the projective plane. We also describe the cone of Betti tables of all finite length modules.
\end{abstract}

\section{Introduction}

The idea of studying the cone of Betti tables of finitely generated graded modules over a polynomial ring originates in the work of Boij and S\"oderberg \cite{boij-soderberg}. To summarize, they conjectured that the cone of Betti tables coming from Cohen--Macaulay modules (of a fixed codimension) is spanned by Betti tables coming from pure resolutions. This conjecture was proven by Eisenbud and Schreyer \cite{es:bs} and later improved to get a statement for the cone of all finitely generated modules \cite{bs2}. For a survey of these developments, we point the reader to \cite{es:survey, floystad}. 

A crucial step was the construction of the defining inequalities of the cone of Betti tables. Eisenbud and Schreyer found a surprising connection with the cone of cohomology tables of vector bundles on projective space. A conceptual understanding of this connection is proposed by Eisenbud and Erman \cite{eisenbud-erman}.

In a different direction, one can replace the polynomial ring with a singular ring and try to generalize these results. The main difference is that all finitely generated modules over polynomial rings have finite length free resolutions, whereas this property fails for singular rings. The next class of rings in terms of complexity of resolutions are the hypersurface rings: the minimal free resolutions of finitely generated modules become periodic (of period $2$). The cone of Betti tables over hypersurface rings of low embedding dimension are described in \cite{BBEG}.

Another perspective, when the ring is Cohen--Macaulay, is that the long-term behavior of a minimal free resolution is reduced to the study of minimal free resolutions of maximal Cohen--Macaulay modules. From this point of view, the next level of complexity of rings are the Cohen--Macaulay rings with finitely many isomorphism classes of indecomposable Cohen--Macaulay modules. These have been classified in \cite{eisenbud-herzog}; one family are the homogeneous coordinate rings of rational normal curves, and the cone of Betti tables over this ring have been described in \cite{kummini-sam}.

In this paper, we describe the cone of Betti tables over another example of such a Cohen--Macaulay ring: the homogeneous coordinate ring of $3$ non-collinear points in the projective plane. Up to isomorphism, these points are $[1:0:0]$, $[0:1:0]$, and $[0:0:1]$. A pleasant feature of this example is that all of the defining inequalities of the cone of Betti tables have concrete interpretations. It is similar to the examples considered in \cite{BBEG}, but one crucial difference is that our ring is not a hypersurface ring, so that this example bridges ideas from \cite{BBEG} and \cite{kummini-sam}.

\medskip

Now we fix notation. Throughout, $\bk$ is a fixed field. The Cohen--Macaulay ring in question, i.e., the coordinate ring of the $3$ points $[1:0:0]$, $[0:1:0]$, $[0:0:1]$, is $B = \bk[x,y,z]/(xy,yz,xz)$.

Let $\bV$ be the $\bQ$-vector space of tables $(v_{i,j})$ where $i \ge 0$ and $j \in \bZ$ with the property that for each $i$, there are only finitely many $j$ such that $v_{i,j}\ne 0$.
Given a finitely generated graded $B$-module $M$, define $\beta^B(M)_{i,j} = \dim_\bk \Tor^B_i(M, \bk)_j$, which is an element of $\bV$. Our main object of study is the $\bQ_{\ge 0}$-span of $\beta^B(M)$ in $\bV$ as $M$ varies over all finitely generated (respectively, finite length) $B$-modules. Call these cones $\rB_\bQ(B)$ and $\rB^f_\bQ(B)$, respectively.

The main results of this paper are Theorem~\ref{thm:main} and Corollary~\ref{cor:main} which completely describe the cone of Betti tables of all finitely generated $B$-modules, and the cone of Betti tables of all finite length $B$-modules, respectively. We give two descriptions: one in terms of generators, and one in terms of linear inequalities.

In \S\ref{sec:general} we begin with some general results on MCM modules and pure resolutions, and in \S\ref{sec:main} we prove the main results and give a local version of the results.

\subsection*{Acknowledgements}

The software system Macaulay2 \cite{M2} was helpful for carrying out this research. In particular, we made use of the FourierMotzkin package written by Gregory Smith.
Iulia Gheorghita was supported by the Caltech SURF program.
Steven Sam was supported by a Miller research fellowship.

\section{Generalities on $B$-modules} \label{sec:general}

\subsection{Notation}

Given a $B$-module $M$, its Hilbert series $\rH_M(t) = \sum_d (\dim_\bk M_d) t^d$ is of the form $p(t)/(1-t)$ for some polynomial $p(t)$. We define $e(M) = p(1)$. From generalities on Hilbert polynomials, we get $\dim_\bk M_d = e(M)$ for $d \gg 0$. In particular, $e(M) \ge 0$, and $M$ is of finite length if and only if $e(M)=0$. 

The (reduced) syzygy module $\Omega(M)$ of $M$ is the kernel of a surjection $\bF \to M \to 0$ where $\bF$ is a free module whose basis maps to a minimal generating set of $M$.

Finally, if $M$ is a graded module, then $M(-d)$ denotes the same module with a grading shift: $M(-d)_e = M_{e-d}$.

\subsection{MCM modules} \label{sec:MCM}

There are $8$ indecomposable maximal Cohen--Macaulay (MCM) $B$-modules (see \cite[Remark 9.16]{yoshino} for the completion of $B$; to compare this to the graded case we can use \cite[Lemma 15.2.1]{yoshino}): the free module $B$, the canonical module $\omega_B$, and the quotients
\[
M_1 = B/(x), \qquad M_2 = B/(y), \qquad M_3 = B/(z),
\]
\[
M_{1,2} = B/(x,y), \qquad M_{2,3} = B/(y,z), \qquad M_{1,3} = B/(x,z).
\]
Their Hilbert series are as follows:
\[
\rH_B(t) = \frac{1+2t}{1-t}, \qquad \rH_{\omega_B}(t) = \frac{2+t}{1-t}, \qquad \rH_{M_i}(t) = \frac{1+t}{1-t}, \qquad \rH_{M_{i,j}}(t) = \frac{1}{1-t}.
\]
An MCM module $M$ is {\bf pure} if $M$ is (up to a grading shift) of one of the following four forms: 
\[
B^{\oplus r}, \qquad \omega_B^{\oplus r}, \qquad M_1^{\oplus r_1} \oplus M_2^{\oplus r_2} \oplus M_3^{\oplus r_3}, \qquad M_{1,2}^{\oplus r_1} \oplus M_{1,3}^{\oplus r_2} \oplus M_{2,3}^{\oplus r_3}.
\]
Implicit in this definition is that $M$ is generated in a single degree. 

We treat each of the modules $M_i$ the same because their numerical invariants (such as Betti numbers) are the same, and similarly with $M_{i,j}$. We use the convention that $M_i$ may refer to any of $M_1, M_2, M_3$ and similarly for $M_{i,j}$. To clarify, if we say that $M$ is a direct sum of $M_i$, we mean that it is a direct sum of copies of $M_1, M_2, M_3$, and similarly for $M_{i,j}$.

\begin{proposition} \label{prop:syzygy}
The syzygy modules are as follows:
\begin{align*}
\Omega(\omega_B) &= (M_{1,2} \oplus M_{2,3} \oplus M_{1,3})(-1),\\
\Omega(M_i) &= M_{\{1,2,3\} \setminus i}(-1),\\
\Omega(M_{i,j}) &= (M_{j,k} \oplus M_{i,k})(-1),
\end{align*}
where in the last line, $k = \{1,2,3\} \setminus \{i,j\}$.
\end{proposition}

\begin{proof}
Let $A = \bk[x,y,z]$. The minimal free resolution of $B$ over $A$ is
\[
0 \to A^2 \xrightarrow{\begin{pmatrix} -z & 0 \\ y & -y \\ 0 & x \end{pmatrix}} A^3 \xrightarrow{\begin{pmatrix} xy & xz & yz \end{pmatrix}} A \to B \to 0.
\]
In particular, $\omega_B$ has the following minimal presentation over $A$ (and $B$):
\[
A^3 \xrightarrow{\begin{pmatrix} -z & y & 0 \\ 0 & -y & x \end{pmatrix}} A^2 \to \omega_B \to 0.
\]
The image of the $3$ basis vectors of $A^3$ give, respectively, copies of $M_{1,2}$, $M_{1,3}$, and $M_{2,3}$.

The last two equalities are straightforward calculations.
\end{proof}

\subsection{Herzog--K\"uhl equations} \label{sec:HK}

\begin{definition}
A $B$-module $M$ has a {\bf pure resolution} if it is Cohen--Macaulay, and there is an exact sequence of the form
\[
0 \to \Omega(M) \to \bF_0 \to M \to 0
\]
where $\Omega(M)$ is a pure MCM module and $\bF_0$ is a free $B$-module generated in a single degree. If $M$ is a finite length module, we say that its type is $(d_0,d_1)$ where $d_0$ is the common degree of all generators of $\bF_0$ and $d_1$ is the common degree of all generators of $\Omega(M)$. If $M$ is a pure MCM module, then its degree sequence is simply $(d_0)$ where $d_0$ is the common degree of its generators.
\end{definition}

The Betti numbers of a pure resolution satisfy certain relations which we now describe. Over a polynomial ring, such relations were worked out by Herzog and K\"uhl \cite[Theorem 1]{herzog-kuhl}. If $M$ is a finite length module with pure resolution of type $(d_0, d_1)$ with $d_0 < d_1$, then it has an exact sequence of the form 
\[ 
0 \rightarrow N(-d_1)^{\beta_1} \rightarrow B(-d_0)^{\beta_0} \rightarrow M \rightarrow 0
\] 
where $N$ is an MCM module: $B$, $\omega_B$, $M_i$, or $M_{i, j}$. The latter three modules have the following minimal free resolutions over $B$:
\begin{align*}
\cdots \rightarrow B(-m)^{3 \cdot 2^{m-1}}\rightarrow \cdots \rightarrow B(-2)^6 \rightarrow B(-1)^3 \rightarrow B^2\rightarrow \omega_B \rightarrow 0,\\
\cdots  \rightarrow B(-m)^{2^{m-1}} \rightarrow \cdots \rightarrow B(-2)^2 \rightarrow B(-1) \rightarrow B \rightarrow M_i \rightarrow 0,\\
\cdots \rightarrow B(-m)^{2^m} \rightarrow \cdots \rightarrow B(-2)^4 \rightarrow B(-1)^2 \rightarrow B \rightarrow M_{i,j}\rightarrow 0.
\end{align*}
To prove this, one can use Proposition~\ref{prop:syzygy}. In particular, in all cases, if 
\[
\cdots \to \bF_i \to \bF_{i-1} \to \cdots \to \bF_1 \to \bF_0 \to M \to 0
\]
is the minimal free resolution of $M$ over $B$, then each $\bF_i$ is generated in a single degree $d_i$; let $\beta_i$ be its rank. We see that $d_i = d_{i-1} + 1$ for $i \ge 2$. Also, for $i \ge 2$, we get a relation 
\begin{align} \label{eqn:higher-betti}
\beta_i = 2^{i-2} c(N) \beta_1
\end{align}
where $c(N)$ depends on the isomorphism type of $N$ and is $0,1,\frac{3}{2},2$ if $N$ is $B, M_i, \omega_B, M_{i,j}$, respectively.

\begin{proposition} \label{prop:HK}
Let $M$ be a finite length $B$-module with a pure resolution of type $(d_0,d_1)$. Write $\Omega(B) = N^{\oplus r}$ where $N$ is one of $B, M_i, \omega_B, M_{i,j}$ and set $c = c(N) \in \{0,1,\frac{3}{2},2\}$. Then the Betti numbers of $M$ are a multiple of the entries of the vector
\[
(1, \frac{3}{3-c}, \frac{3c}{3-c}, 2 \frac{3c}{3-c}, 4 \frac{3c}{3-c}, 8 \frac{3c}{3-c}, \dots, 2^{i-2} \frac{3c}{3-c}, \dots)
\]
\end{proposition}

\begin{proof}
Using \eqref{eqn:higher-betti}, we get
\begin{align*}
\rH_M(t) &= \rH_B(t) (\beta_0t^{d_0} - \beta_1t^{d_1} + \beta_1c t^{d_1 + 1} \sum_{i \geq 2}(-1)^i2^{i-2}t^{i-2} ) \\
&=\frac{1+2t}{1-t} (\beta_0t^{d_0} - \beta_1t^{d_1} + \frac{\beta_1 c t^{d_1 + 1}}{1+2t} ) \\ 
&= \frac{(1+2t)(\beta_0t^{d_0} -\beta_1t^{d_1}) + \beta_1 c t^{d_1 + 1} }{1-t}.
\end{align*} 
Since $M$ is finite length, $\rH_M(t)$ is a polynomial, and so the numerator of the above expression is divisible by $1-t$. So the substitution $t \mapsto 1$ gives the relation $3(\beta_0 - \beta_1) + \beta_1c = 0$. Taking $\beta_0=1$ gives us the desired statement.
\end{proof}

\begin{remark} \label{rmk:convex}
The Betti numbers do not depend on the actual degree sequence, just on the isomorphism class of the MCM module. For convenience, we list here the first lattice point on the rays spanned by the vectors from Proposition~\ref{prop:HK}:
\begin{align*}
\bv_1 &= (1, 1, 0, 0, 0, \dots),\\
\bv_2 &= (2, 3, 3, 6, 12, \dots),\\
\bv_3 &= (1, 2, 3, 6, 12, \dots),\\
\bv_4 &= (1, 3, 6, 12, 24, \dots).
\end{align*}
Note that $2\bv_3 = \bv_1 + \bv_4$ and $2\bv_2 = 3\bv_1 + \bv_4$, so that the cone generated by $\bv_1, \bv_2, \bv_3, \bv_4$ is minimally generated by $\bv_1$ and $\bv_4$.
\end{remark}

\section{The cone of Betti diagrams} \label{sec:main}

\subsection{Main result}

Each pure resolution has an associated degree sequence which documents in what degree each term of the resolution is generated. If the resolution is finite, we pad the degree sequence with the symbol $\infty$, so the possible degree sequences are
\[
(d_0, \infty, \infty, \dots), \qquad (d_0, d_1, \infty, \dots), \qquad (d_0, d_1, d_1 + 1, d_1 + 2, \dots)
\]
where $d_0 < d_1$. 

For each degree sequence $d$ define $\pi_d \in \bV$ as follows. First, $(\pi_d)_{0,j} = 1$ for $j = d_0$ and $0$ otherwise. If $d = (d_0, \infty, \infty, \dots)$, $\pi_{i,j} = 0$ for all $i\geq 1$ and all $j$. If $d =(d_0, d_1, \infty, \dots)$, $\pi_{1,j} = 1$ for $j = d_1$ and $0$ otherwise. If $d = (d_0, d_1, d_1 + 1, d_1 + 2, \dots)$, $\pi_{i,j} = 3 \cdot 2^{i-1}$ for $i\geq 1$ and $j = d_i$ and $0$ otherwise. Recall that these are the Betti numbers (up to scalar multiple) of a finite length module whose syzygy module is a direct sum of copies of $M_{i,j}$ and that we do not use the others because of the relations in Remark~\ref{rmk:convex}.

For each $v \in \bV$, define 
\[
\epsilon_{i,j}(v) = v_{i,j}, \quad \alpha_k(v) = 2\epsilon_{1,k}(v) - \epsilon_{2, k+1}(v), \quad \gamma_k(v) = \sum_{j \leq k}(3\epsilon_{0,j}(v) - 3\epsilon_{1, j+1}(v) + \epsilon_{2, j+2}(v)).
\]
We also define $\gamma_\infty$ by summing over all $j$. Note that while $\gamma_k$ involves an infinite sum, only finitely many nonzero terms of $v$ are involved so it is well-defined.

\begin{theorem} \label{thm:main}
The following cones are equal:
\begin{enumerate}[\indent \rm (i)]
\item The cone $\rB_{\bQ}(B)$ spanned by the Betti diagrams of all finitely generated $B$-modules.
\item The cone $D$ spanned by $\pi_d$ for all $B$-degree sequences $d$.
\item The cone $F$ defined to be the intersection of the halfspaces $\{\epsilon_{i, j} \geq 0\}$ for all $i,j \geq 0$, $\{\alpha_k \geq 0\}$ for all $k \in \bZ$, $\{\gamma_k \geq 0\}$ for all $k \in \bZ$, and the subspaces $\{2\epsilon_{i,j} = \epsilon_{i+1,j+1}\}$ for $i \geq 2, j \in \bZ$.
\end{enumerate}
\end{theorem}

The proof of Theorem~\ref{thm:main} follows by establishing the inclusions $D \subseteq \mathrm{B}_{\bQ}(B) \subseteq F \subseteq D$, which are the content of the next three lemmas.

\begin{lemma}
$D \subseteq \mathrm{B}_{\bQ}(B)$.
\end{lemma}

\begin{proof}
It suffices to show that for each degree sequence $d$ there exists a $B$-module $M$ with $\beta^B(M) = \pi_d$. We assume $d_0=0$ for simplicity of notation. 
\begin{itemize}[$\bullet$]
\item When $d = (0, \infty, \infty, \dots)$, take $M = B$. 
\item When $d = (0, d_1, \infty, \dots)$, take $M = B/( (x+y+z)^{d_1} )$. 
\item When $d = (0, d_1, d_1 +1, d_1 + 2, \dots)$, take $M = B/(x^{d_1}, y^{d_1}, z^{d_1})$. \qedhere
\end{itemize}
\end{proof}

\begin{lemma} \label{lem:BF}
$\mathrm{B}_{\bQ}(B) \subseteq F$.
\end{lemma}

\begin{proof}
For a finitely generated graded $B$-module $M$, we want to show that the inequalities defining $F$ are nonnegative on $\beta^B(M)$. There are no negative entries in $\beta^B(M)$, so $\epsilon_{i,j}(\beta^B(M)) = \beta_{i,j}^B(M) \geq 0$ for all $i,j$. The inequalities $\{\alpha_k \geq 0 \mid k \in \bZ\}$ and the equalities $\{ 2\epsilon_{i,j} = \epsilon_{i+1,j+1} \mid i \ge 2,\ j \in \bZ \}$ hold for $\beta^B(M)$ by our discussion in \S\ref{sec:HK}.

It remains to show that $\gamma_k(\beta^B(M)) \ge 0$ for all $k \in \bZ$. Let $\cdots \to \bF_1 \to \bF_0 \to M \to 0$ be a minimal free resolution of $M$. Suppose that $\Omega(M) \cong B^{\oplus \alpha_0} \oplus \omega_B^{\oplus \alpha_1} \oplus M_i^{\oplus \alpha_2} \oplus M_{i,j}^{\oplus \alpha_3}$. Let $\beta_i$ be the $i$th Betti number of $M$ (ignoring the grading). From \S\ref{sec:MCM} and \S\ref{sec:HK}, we can read off the relations
\begin{align*}
e(\Omega(M)) &= 3\alpha_0 + 3\alpha_1 + 2\alpha_2 + \alpha_3,\\
\beta_1 &= \alpha_0 + 2\alpha_1 + \alpha_2 + \alpha_3,\\
\beta_2 &= 3\alpha_1 + \alpha_2 + 2\alpha_3.
\end{align*}
In particular, $e(\Omega(M)) = 3\beta_1 - \beta_2$. From the short exact sequence 
\[
0 \to \Omega(M) \to \bF_0 \to M \to 0,
\]
we get $e(M) = e(\bF_0) - e(\Omega(M)) = 3\beta_0 - 3\beta_1 + \beta_2$, and $e(M)$ is a nonnegative quantity, so 
\begin{align} \label{eqn:mult-ineq}
3 \beta_0 - 3\beta_1 + \beta_2 \ge 0.
\end{align} 

If $\bF$ is a free $B$-module, let $\tau_{\le k}(\bF)$ be the free module generated by minimal generators of degree $\le k$. For a general $B$-module $M$, define $\tau_{\le k}(M)$ to be the quotient of $\tau_{\le k+1}(\bF_1) \to \tau_{\le k}(\bF_0)$. In \S\ref{sec:HK}, we established that the differentials $\bF_i \to \bF_{i-1}$ are linear for $i \ge 2$, and so we conclude that the following is a minimal free resolution of $\tau_{\le k}(M)$:
\[
\cdots \to \tau_{\le k+i}(\bF_i) \to \tau_{\le k+i-1}(\bF_{i-1}) \to \cdots \to \tau_{\le k+1}(\bF_1) \to \tau_{\le k}(\bF_0) \to \tau_{\le k}(M) \to 0.
\]
In particular, applying \eqref{eqn:mult-ineq} to $\tau_{\le k}(M)$, we conclude that $\gamma_k(\beta^B(M)) \ge 0$, as desired.
\end{proof}

\begin{lemma}
$F \subseteq D$.
\end{lemma}

\begin{proof}
The proof is formally the same as the proofs of \cite[Lemmas 2.7, 2.8]{BBEG}.
\end{proof}

\begin{remark}
Following \cite{BBEG}, define a partial order on the degree sequences by $d \leq d'$ if $d_0 \leq d_0'$ and $d_1 \leq d_1'$ where either inequality is strict, or if $d_0 = d_0', d_1 = d_1',$ and $d_n \leq d_n'$ for all $n \geq 2$. Then the proofs of \cite[Lemmas 2.7, 2.8]{BBEG} show that $\rB_\bQ(B)$ has a triangulation coming from the simplicial cones spanned by the rays corresponding to the elements of maximal chains in this partial order.
\end{remark}

From Theorem~\ref{thm:main}, we immediately get a result about the cone of finite length $B$-modules:

\begin{corollary} \label{cor:main}
The following cones are equal:
\begin{enumerate}[\indent \rm (i)]
\item The cone $\rB^f_{\bQ}(B)$ spanned by the Betti diagrams of all finite length $B$-modules.
\item The cone $D^f$ spanned by $\pi_d$ for all $B$-degree sequences $d$ with $d_1 < \infty$.
\item The cone $F^f$ defined to be the intersection of the halfspaces $\{\epsilon_{i, j} \geq 0\}$ for all $i,j \geq 0$, $\{\alpha_k \geq 0\}$ for all $k \in \bZ$, $\{\gamma_k \ge 0\}$ for all $k \in \bZ$, and the subspaces $\{2\epsilon_{i,j} - \epsilon_{i+1,j+1} = 0 \}$ for $i \geq 2, j \in \bZ$ and $\gamma_\infty = 0$.
\end{enumerate}
\end{corollary}

\begin{proof}
In each of three items, we have added one extra condition to the cones defined in Theorem~\ref{thm:main}, so we just have to verify that these conditions are the same. From the proof of Lemma~\ref{lem:BF}, we get that $\gamma_\infty(\beta^B(M)) = e(M)$. Since $e(M)=0$ if and only if $M$ has finite length, this shows that the extra conditions in (a) and (c) coincide. From Proposition~\ref{prop:HK}, $\gamma_\infty(\pi_d)=0$ if $d_2 < \infty$. If $d_1 < \infty$ and $d_2=\infty$, it is immediate from the definition that $\gamma_\infty(\pi_d)=0$. Finally, for the last case with $d_1=\infty$, we have $\gamma_\infty(\pi_d)=1$, so the extra conditions in (b) and (c) coincide.
\end{proof}

\subsection{Local version}

Following \cite{BEKS}, we now give a local version of our main result by considering minimal free resolutions over the completion $\hat{B} =\bk[\![x,y,z]\!]/(xy,yz,xz)$. In this case, the relevant numerical invariant is the Betti sequence $\beta(M)$ of a module $M$ which records the number of generators of each term in a minimal free resolution of $M$.

As discussed above, the only relevant information are the first $3$ Betti numbers $(\beta_0, \beta_1, \beta_2)$ since $\beta_i = 2\beta_{i-1}$ for $i > 2$. The following results are easy to prove from what we have already discussed.

\begin{proposition}
The following cones are equal:
\begin{enumerate}[\indent \rm (i)]
\item The cone spanned by the Betti sequences of all finitely generated $\hat{B}$-modules.
\item The cone spanned by $(1,0,0,\dots)$, $(1,1,0,\dots)$, and $(1,3,6,\dots)$.
\item The cone defined by the inequalities $\beta_2 \ge 0$, $3\beta_0 + \beta_2 \ge 3\beta_1$, and $2\beta_1 \ge \beta_2$, and the equalities $\beta_i = 2\beta_{i-1}$ for $i > 2$.
\end{enumerate}
\end{proposition}

\begin{proposition}
The following cones are equal:
\begin{enumerate}[\indent \rm (i)]
\item The cone spanned by the Betti sequences of all finite length $\hat{B}$-modules.
\item The cone spanned by $(1,1,0,\dots)$, and $(1,3,6,\dots)$.
\item The cone defined by the inequalities $\beta_2 \ge 0$ and $2\beta_1 \ge \beta_2$, and the equalities $3\beta_0 + \beta_2 = 3\beta_1$ and $\beta_i = 2\beta_{i-1}$ for $i > 2$.
\end{enumerate}
\end{proposition}

\end{document}